\documentclass[12pt]{amsart}

 \usepackage{amsfonts,graphics,amsmath,amsthm,amsfonts,amscd,amssymb,amsmath,latexsym,multicol,
 mathrsfs}
\usepackage{epsfig,url}
\usepackage{flafter}
\usepackage{fancyhdr}
\usepackage{hyperref}
\hypersetup{colorlinks=true, linkcolor=black}

\makeatletter

\def\jobis#1{FF\fi
  \def\predicate{#1}%
  \edef\predicate{\expandafter\strip@prefix\meaning\predicate}%
  \edef\job{\jobname}%
  \ifx\job\predicate
}

\makeatother

\if\jobis{proposal}%
\else
\fi

 \usepackage[matrix, arrow]{xy}

\DeclareMathOperator{\Supp}{Supp}

\DeclareMathOperator{\Fix}{Fix}

 \numberwithin{equation}{subsection}
 \numberwithin{footnote}{subsection}

 \newtheorem{cor}[subsection]{Corollary}
 \newtheorem{lem}[subsection]{Lemma}
 \newtheorem{prop}[subsection]{Proposition}
 \newtheorem{thm}[subsection]{Theorem}

{
    \newtheoremstyle{upright}%
        {8pt plus2pt minus4pt}%
        {8pt plus2pt minus4pt}%
        {\upshape}%
        {}%
        {\bfseries\scshape}%
        {}%
        {1em}%
        {}%
\theoremstyle{upright}

 \newtheorem{exa}[subsection]{Example}

 \newtheorem{rem}[subsection]{Remark}

}

 \newcommand{\Q}{\mathbb Q}
 \newcommand{\R}{\mathbb R}
 
 \newcommand{\bir}{\dashrightarrow}
 \newcommand{\rddown}[1]{\left\lfloor{#1}\right\rfloor} 

\title{\large L\MakeLowercase{og canonical pairs with good augmented base loci}}
\author{\large C\MakeLowercase{aucher} B\MakeLowercase{irkar and} Z\MakeLowercase{hengyu} H\MakeLowercase{u}}
\date{\today}
\begin{document}
\maketitle

\begin{abstract}
Let $(X,B)$ be a projective log canonical pair such that $B$ is a $\Q$-divisor, and that 
there is a surjective morphism $f\colon X\to Z$ onto a normal variety $Z$ satisfying: 
$K_X+B\sim_\Q f^*M$ for some big $\Q$-divisor $M$, and the augmented base locus 
${\bf{B_+}}(M)$ does not contain the 
image of any log canonical centre of $(X,B)$.
We will show that $(X,B)$ has a good log minimal model. An interesting special case
is when $f$ is the identity morphism.
\end{abstract}



\section{Introduction}

\textbf{Main results of this paper.}
We work over an algebraically closed field $k$ of characteristic zero.
For simplicity we will prove our results in the absolute projective case 
but they can be formulated and proved similarly in the relative case.

For a $\Q$-divisor $M$ on a normal projective variety $Z$, the stable base locus is 
denoted by ${\bf{B}}(M)$ and the augmented base locus by ${\bf{B_+}}(M)$. 
The latter is defined as ${\bf{B_+}}(M)=\bigcap_A {\bf{B}}(M-A)$ 
where the intersection runs over all ample $\Q$-divisors $A$.

Concerning the augmented base loci of divisors related to log canonical (lc) pairs 
we have the following statement which is one 
of the main results of this paper. 

\begin{thm}\label{t-main}
Let $(X,B)$ be a projective lc pair such that $B$ is a $\Q$-divisor, and that 
there is a surjective morphism $f\colon X\to Z$ onto a normal projective variety $Z$ satisfying: 

$\bullet$ $K_X+B\sim_\Q f^*M_Z$ for some big $\Q$-divisor $M_Z$,

$\bullet$ ${\bf{B_+}}(M_Z)$ does not contain the image of any lc centre of $(X,B)$.\\\\
Then, $(X,B)$ has a good log minimal model. In particular, the log canonical algebra 
$R(K_X+B)$ is finitely generated over $k$. 
\end{thm}

The proof of the theorem is given in Section 5. The main difficulties of the proof are 
due to presence of lc singularities.  Perhaps 
it is a good place to emphasize that trying to understand lc pairs rather than Kawamata log terminal (klt) pairs 
(or just smooth varieties) is not simply for the sake of generality. It is often the case that 
failure to prove a statement 
about lc pairs of dimension $d$ comes from failure to understand certain aspects of 
smooth varieties in dimension $d-1$. For example if we cannot prove finite generation of 
lc rings of lc pairs of dimension $d$ it is because we do not know how to prove abundance 
for varieties of dimension $d-1$. To be more precise, finite generation of lc rings of 
lc pairs of dimension $d$ is
equivalent to the existence of good minimal models of $\mathbb Q$-factorial pseudo-effective 
dlt pairs of dimension $\leq d-1$ ( see Fujino-Gongyo [\ref{Fujino-Gongyo-lc-rings}] for more details).

\begin{cor}\label{cor-1}
Let $(X,B)$ be a projective lc pair such that $K_X+B$ is a big $\Q$-divisor and 
that ${\bf{B_+}}(K_X+B)$ 
does not contain any lc centre of $(X,B)$. Then, $(X,B)$ has a good log minimal model. 
In particular, the log canonical algebra $R(K_X+B)$ is finitely generated over $k$. 
\end{cor}

The corollary follows immediately from Theorem \ref{t-main}. 
The klt case of the 
corollary follows easily from [\ref{BCHM}]. However, passing from klt to lc is often  
very subtle as pointed above.

One may wonder whether the 
corollary still holds if we assume that $(X,B)$ is log big instead of assuming that 
${\bf{B_+}}(K_X+B)$ does not contain lc centres of $(X,B)$. Here log big means 
that $K_X+B$ is big and $(K_X+B)|_S$ is also big for any lc centre $S$.
Example \ref{exa-log-big} below shows that one gets into serious trouble very quickly.

The corollary implies the following result which was conjectured 
in [\ref{Cacciola}] and first proved in 
[\ref{BH}, Theorem 1.6]: both papers put the extra assumption that $K_X+B+P$ birationally 
has a CKM-Zariski decomposition. 
Recall that a lc polarized pair $(X,B+P)$ consists 
of a lc pair $(X,B)$ together with a nef divisor $P$.

\begin{cor}\label{cor-2}
Let $(X,B+P)$ be a projective lc polarized pair such that $B,P$ are $\Q$-divisors, 
$K_X+B+P$ is big, and 
that the augmented base locus ${\bf{B_+}}(K_X+B+P)$ 
does not contain any lc centre of $(X,B)$. Then, the log canonical algebra $R(K_X+B+P)$
 is finitely generated over $k$.\\ 
\end{cor}

{\textbf{Sketch of the proof of Theorem \ref{t-main}.}} Roughly speaking we follow the ideas of 
the proof of [\ref{B-lc-flips}, Theorem 1.4]. The flexibility allowed in the statement of 
Theorem \ref{t-main} enables us to do induction on dimension. By comparison it seems hard to 
apply such arguments to prove Corollary \ref{cor-1}.  
Let $(X,B)$ and $f$ be as in the theorem. 
We can take a dlt blowup hence assume that $(X,B)$ is $\Q$-factorial dlt. 
If  there is a 
component $S$ of $\rddown{B}$ such that $S$ is mapped onto $Z$, then we can 
use induction to show that the algebra of $(K_X+B)|_S$ is finitely generated 
and from this we can derive finite generation of $R(K_X+B)$. Once we have this finite 
generation we can find a good log minimal model of $(X,B)$ (see Proposition \ref{p-fg}).
We can then assume that no component of $\rddown{B}$ is mapped onto $Z$.

By assumptions, we can write $M_Z\sim_\Q A_Z+L_Z$ where $A_Z\ge 0$ is ample and $L_Z\ge 0$, 
and $\Supp (A_Z+L_Z)$ does not contain the image of any lc centre of $(X,B)$. 
Let $A=f^*A_Z$ and $L=f^*L_Z$. After replacing $(X,B)$ and $f$ we can write 
$A\sim_\Q G+P$ where 
$G\ge 0$ is semi-ample containing no lc centre of $(X,B)$, $P\ge 0$, 
and $\Supp P=\Supp \rddown{B}$. Moreover, 
$K_X+B+ b G+aL$ is nef if $b-a\gg 0$ and $a\ge 0$ (see Lemma \ref{l-1}). 
 
Let $\epsilon_1>\epsilon_2>\cdots $ be a sequence of sufficiently small rational 
numbers with $\lim \epsilon_j=0$. For each $\epsilon_j$, 
we can run an LMMP on 
$K_X+B+\epsilon_j G+\epsilon_j L$ with scaling of $\alpha G$ for some large $\alpha$ such that 
the LMMP ends up with a good log minimal model $X_j'$ (see Lemma \ref{l-2}). We can make sure that 
the $X_j'$ are all isomorphic in codimension one. 

We can run an LMMP on $K_{X_1'}+B_1'$ 
with scaling of $\epsilon_1 G_1'+\epsilon_1 L_1'$ such that 
$\lim \lambda_i=0$ where $\lambda_i$ are the numbers appearing in the LMMP (see Lemma \ref{l-4}). 
Moreover, after some modifications and redefining the notation 
we can assume that  each $X_j'$ appears in some step of the latter LMMP. 
By applying special termination and induction we show that the LMMP terminates 
near $\rddown{B_1'}$ (see Lemma \ref{l-3}). This implies that the LMMP terminates 
everywhere because of the choice of $P,G,L$. This essentially gives the minimal model 
we are looking for because of our choice of the $\epsilon_j$ and the construction.\\

\textbf{Organization of the paper.} In Section 2, we bring together some basic definitions, 
conventions, and notation. In Section 3, we prove the existence of log minimal models 
for certain pairs. We will need these for the arguments presented in the later sections. In Section 4,  
we prove several lemmas in order to prepare for the proof of Theorem \ref{t-main}. 
Finally in Section 5, we give the proof of \ref{t-main}, \ref{cor-1}, and \ref{cor-2}.\\

\textbf{Acknowledgements.} The first author was supported by a Leverhulme grant, and the second 
author was supported by an EPSRC grant. We would like to thank the referee for the valuable corrections 
and suggestions.

\vspace{0.5cm}
\section{Preliminaries}

Let $k$ be an algebraically closed field of characteristic zero fixed throughout the paper. 
All the varieties will be over $k$ unless stated otherwise.\\ 

\subsection{Divisors.} 
First we introduce some \emph{notation}. If $X\bir X'$ is a birational map between normal projective varieties  
whose inverse does not contract divisors: if $M$ is an $\R$-divisor on $X$ 
we denote its birational transform on $X'$ by $M'$ or by $M_{X'}$.  
Another notation we use is when we have a birational map $X_1\bir X_2$ 
whose inverse does not contract divisors: if $M_1$ is a divisor on $X_1$ 
we denote its birational transform on $X_2$ by $M_2$. 
All these notations will be clear from the context.

Now let $X$ be a normal projective variety and $M$ a $\Q$-divisor on $X$. 
The \emph{stable base locus} of $M$ is defined as ${\bf{B}}(M)=\bigcap_N \Supp N$ 
where $N$ ranges over all effective $\Q$-divisors satisfying $N\sim_\Q M$.  
On the other hand, the \emph{augmented base locus} of $M$ is defined as
$$
{\bf{B_+}}(M)=\bigcap_A {\bf{B}}(M-A)
$$ 
where $A$ ranges over all ample $\Q$-divisors.
It is not difficult to see that ${\bf{B_+}}(M)={\bf{B}}(M-\epsilon A)$ for some 
sufficiently small $\epsilon>0$ if we fix the ample divisor $A$.

The \emph{divisorial algebra} associated to $M$ is defined as 
$$
{R}(X,M)=\bigoplus_{m\ge 0}H^0(X,\rddown{mM})
$$

For a given surjective morphism $f\colon X\to Z$ we say that $M$ is \emph{very exceptional} if 
$M$ is vertical$/Z$, that is, $\Supp M$ is not mapped onto $Z$, and  
if for any prime divisor $P$ on $Z$ there is a prime divisor $Q$ on $X$ which is 
not a component of $M$ but $f(Q)=P$.\\  

\subsection{Pairs and polarized pairs.} 
A \emph{pair} $(X,B)$ consists of a normal quasi-projective variety $X$ and an $\R$-divisor $B$ on $X$ with
coefficients in $[0,1]$ such that $K_X+B$ is $\mathbb{R}$-Cartier. In this paper we mostly deal with pairs 
with $B$ being a $\Q$-divisor.
For a prime divisor $D$ on some birational model of $X$ with a
nonempty centre on $X$, $a(D,X,B)$
denotes the log discrepancy. For definitions and standard results on singularities of pairs 
we refer to [\ref{Kollar-Mori}]. 

A \emph{polarized pair} $(X,B+P)$ consists of a projective 
pair $(X,B)$ and a nef $\R$-divisor $P$. We say that $(X,B+P)$ is lc when $(X,B)$ is lc.

A projective pair $(X,B)$ is called \emph{log big} when $K_X+B$ is big and for any lc centre 
$S$ of $(X,B)$ the pullback of $K_X+B$ to $S^\nu$ is big where $S^\nu\to S$ is the normalization. 
On the other hand we say that a projective pair $(X,B)$ is 
\emph{log abundant} when $K_X+B$ is abundant and for any lc centre 
$S$ of $(X,B)$ the pullback of $K_X+B$ to $S^\nu$ is abundant. Recall that a divisor $M$ is said to be abundant 
when $\kappa(M)=\kappa_\sigma(M)$ where $\kappa_\sigma$ is the numerical Kodaira dimension defined by 
Nakayama. Although one can make sense of this for $\R$-divisors but we only use the notion when 
$M$ is a $\Q$-divisor.\\

\subsection{Log minimal, weak lc, and log smooth models.}\label{d-mws}
A projective pair $(Y,B_Y)$ is a \emph{log birational model} of a projective pair $(X,B)$ if we are given a birational map
$\phi\colon X\bir Y$ and $B_Y=B^\sim+E$ where $B^\sim$ is the birational transform of $B$ and 
$E$ is the reduced exceptional divisor of $\phi^{-1}$, that is, $E=\sum E_j$ where $E_j$ are the
exceptional/$X$ prime divisors on $Y$. 
A log birational model $(Y,B_Y)$ is a  \emph{weak lc model} of $(X,B)$ if

$\bullet$ $K_Y+B_Y$ is nef, and

$\bullet$ for any prime divisor $D$ on $X$ which is exceptional/$Y$, we have
$$
a(D,X,B)\le a(D,Y,B_Y)
$$

 A weak lc model $(Y,B_Y)$ is a \emph{log minimal model} of $(X,B)$ if 

$\bullet$ $(Y,B_Y)$ is $\Q$-factorial dlt,

$\bullet$ the above inequality on log discrepancies is strict.\\

Let $(X,B)$ be a lc pair, and let $f\colon W\to X$ be a log resolution. 
 Let $B_W\ge 0$ be a boundary on $W$ so that 
$$
K_W+B_W=f^*(K_X+B)+E
$$ 
where $E\ge 0$ is exceptional$/X$ and the support of $E$ contains  
each prime exceptional$/X$ divisor $D$ on $W$ if $a(D,X,B)>0$. 
We call $(W/Z,B_W)$ a \emph{log smooth model} of $(X/Z,B)$. Note that the coefficients of 
the exceptional$/X$ prime divisors in $B_W$ are not necessarily $1$.\\

\subsection{LMMP with scaling.}
Let $(X_1,B_1+C_1)$ be a lc pair such that $(X_1,B_1)$ is $\Q$-factorial dlt, 
$K_{X_1}+B_1+C_1$ is nef, and $C_1\ge 0$. 
Now, by [\ref{B-m-model}, Lemma 3.1], either $K_{X_1}+B_1$ is nef or there is an extremal ray $R_1$ such
that $(K_{X_1}+B_1)\cdot R_1<0$ and $(K_{X_1}+B_1+\lambda_1 C_1)\cdot R_1=0$ where
$$
\lambda_1:=\inf \{t\ge 0~|~K_{X_1}+B_1+tC_1~~\mbox{is nef}\}
$$
If $K_{X_1}+B_1$ is nef or if $R_1$ defines a Mori fibre structure, we stop. 
Otherwise $R_1$ gives a divisorial 
contraction or a log flip $X_1\bir X_2$. We can now consider $(X_2,B_2+\lambda_1 C_2)$  where $B_2+\lambda_1 C_2$ is 
the birational transform 
of $B_1+\lambda_1 C_1$ and continue. 
 By continuing this process, we obtain a sequence of numbers $\lambda_i$ and a 
special kind of LMMP which is called the \emph{LMMP on $K_{X_1}+B_1$ with scaling of $C_1$}. 
Note that by definition $\lambda_i\ge \lambda_{i+1}$ for every $i$, and we usually put 
$\lambda=\lim_{i\to \infty} \lambda_i$.

\vspace{0.5cm}
\section{Minimal models and termination for certain pairs}

In this section, we will prove some results on minimal models and termination 
that we will need for the proof 
of Theorem \ref{t-main}. The arguments in this section are similar to those of 
[\ref{B-lc-flips}, Section 5] but since we cannot directly refer to the results of 
[\ref{B-lc-flips}, Section 5] we will write detailed proofs.

\begin{rem}
In this and later sections, we will apply [\ref{B-lc-flips}, Theorem 1.4] in several places. 
That theorem assumes the ACC for lc thresholds which is by now a theorem of [\ref{HMX}]. 
The ACC is not needed in the klt case.
\end{rem}

\begin{prop}\label{p-fg}
Let $(X,B)$ be a projective lc pair with $B$ a $\Q$-divisor, 
and $f\colon X\bir Z$ a rational map onto a normal projective variety 
$Z$ so that we have:

$\bullet$ $f$ is a projective morphism with connected fibres over some non-empty open subset $U\subseteq Z$,  

$\bullet$ $K_X+B\sim_\Q 0$ over $U$, 

$\bullet$ $\kappa(K_X+B)\ge \dim Z$, and 

$\bullet$ the algebra $R(K_X+B)$ is finitely generated over $k$.\\\\
Then $(X,B)$ has a good log minimal model.
\end{prop}
\begin{proof}
Since $R(K_X+B)$ is a finitely generated $k$-algebra, there is a log resolution  
$g\colon W\to X$, a contraction $h\colon W\to T$,  
and a decomposition $g^*(K_X+B)=A+E$ where $E\ge 0$, $A$ is the pullback of some 
ample $\Q$-divisor on $T$, and for every sufficiently 
divisible integer $m>0$ we have $\Fix(mg^*(K_X+B))=mE$. 

Let $C$ be a rational boundary so that $(W,C)$ is a log smooth model of $(X,B)$ 
(as defined in \ref{d-mws}). 
We can write 
$$
K_W+C=g^*(K_X+B)+F=A+E+F
$$ 
where $F\ge 0$ is exceptional$/X$. In particular, for every sufficiently 
divisible integer $m>0$ we have 
$$
\Fix(m(K_W+C))=mE+mF
$$ 
Moreover, $K_W+C\sim_\Q E+F/T$. 

We may assume that the rational map $W\bir Z$ is a morphism. Since $K_X+B\sim_\Q 0$ over $U$, 
$g^*(K_X+B)\sim_\Q 0$ over $U$. So, $A\sim_\Q 0$ and $E\sim_\Q 0$ on the general fibres of 
$W\to Z$ which implies that such general fibres are contracted to points by $W\to T$. 
Therefore, by comparing the dimensions of the fibres of $W\to Z$ and $W\to T$ and by the 
other assumptions we get  
$$
\dim Z\le \kappa(g^*(K_X+B))=\dim T\le \dim Z
$$
which implies that 
$$
\kappa(K_W+C)=\kappa(g^*(K_X+B))=\dim T=\dim Z
$$ 
Thus, we have an induced birational map 
$\psi \colon Z\bir T$. The birationality comes from the fact that $W\to Z$ and $W\to T$ 
both have connected fibres. Perhaps after shrinking $U$ 
 we can assume that $\psi|_U$ is an isomorphism.

 Now run an LMMP$/T$ on $K_W+C$ with scaling of some ample divisor.  Since $K_X+B\sim_\Q 0$ over $U$, $(W,C)$ has a 
log minimal model over $\psi(U)$ by [\ref{B-lc-flips}, Corollary 3.7] hence the LMMP terminates over $\psi(U)$ 
by [\ref{B-lc-flips}, Theorem 1.9]. So, we arrive at a model $W'$ on which 
$$
K_{W'}+C'\sim_\Q E'+F'\sim_\Q 0
$$
 over $\psi(U)$ (recall that we use $E'$ to denote the birational transform of $E$; similar notation 
 for other divisors). 
In particular, since $E'+F'\ge 0$, this means that 
$E'+F'$ is vertical$/T$. Moreover, since $W\bir W'$ is a partial LMMP on $K_W+C$, for every sufficiently 
divisible integer $m>0$ we have 
$$
\Fix(m(K_{W'}+C'))=mE'+mF'
$$ 
which implies that $E'+F'$ is very exceptional$/T$ by [\ref{B-lc-flips}, Lemma 3.2]. 
Therefore, by [\ref{B-lc-flips}, Theorem 3.4], we can run an LMMP$/T$ on $K_{W'}+C'$ so that  
it contracts $E'+F'$ and so it terminates with a model $W''$ on which $K_{W''}+C''\sim_\Q 0/T$.  
In fact, $K_{W''}+C''\sim_\Q A''$ is semi-ample since $A''$ is the pullback of an ample $\Q$-divisor on $T$.
The pair $(W'',C'')$ is a good log minimal model of $(W,C)$ hence also a good log minimal model of 
$(X,B)$ by [\ref{B-lc-flips}, Remark 2.8].\\ 
\end{proof}

\begin{prop}[{cf. [\ref{Gongyo-Lehmann}]}]\label{p-klt-trivial}
Let $(X,B)$ be a projective klt pair and $f\colon X\to Z$ a contraction such that $K_X+B\sim_\R f^*M_Z$ for 
some big $\R$-Cartier divisor $M_Z$. Then, $(X,B)$ has a good log minimal model.
\end{prop}
\begin{proof}
By applying the canonical bundle formula [\ref{Ambro-adjunction}][\ref{Fujino-Gongyo-adjunction}], 
we can find a boundary $B_Z$ on $Z$ so that $(Z,B_Z)$ is klt and  
$$
K_{X}+B\sim_\R f^*(K_{Z}+B_Z)
$$
By assumptions, $K_{Z}+B_Z$ is big hence by [\ref{BCHM}] $(Z,B_Z)$ has a  
good log minimal model $(T,B_T)$.
Let $M_T$ on $T$ be the pushdown of $M_Z$ which is semi-ample by construction. There is a common 
resolution $d\colon V\to Z$ and $e\colon V\to T$ such that $d^*M_Z-e^*M_T$ is 
effective and exceptional$/T$.

Now take a log resolution $g\colon W\to X$ so that $h\colon W\bir V$ is a morphism, 
and let $G=h^*(d^*M_Z-e^*M_T)$. 
Let $C$ be a boundary so that $(W,C)$ is a klt log smooth model of $(X,B)$ as defined in 
\ref{d-mws}. We can write 
$$
K_W+C=g^*(K_X+B)+F\sim_\R g^*f^*M_Z+F=h^*d^*M_Z+F=h^*e^*M_T+G+F
$$ 
where $F\ge 0$ is exceptional$/X$. 

Let $U\subset Z$ be the largest open subset such that $\psi|_U$ is an isomorphism 
where $\psi$ is the birational map $Z\bir T$. Note that the codimension of 
$T\setminus \psi(U)$ is at least $2$, and $G=0$ over $U$.
Now run an LMMP$/T$ on $K_W+C$ with scaling of some ample divisor. Since $K_X+B\sim_\R 0/Z$, 
the LMMP terminates over $\psi(U)$ 
and $F$ is contracted over $\psi(U)$. So, we arrive at a model $W'$ on which 
$K_{W'}+C'\sim_\R G'+F'/T$ where $G'+F'=0$ over $\psi(U)$. 
Actually, by construction, $G'+F'$ is mapped into $T\setminus \psi(U)$, in particular, $G'+F'$ is very exceptional$/T$. 
 Therefore, by [\ref{B-lc-flips}, Theorem 3.4], we can run an LMMP$/T$ on $K_{W'}+C'$ so that  
it contracts $G'+F'$ and so it terminates with a model $W''$ on which $K_{W''}+C''$ 
is $\R$-linearly equivalent to the pullback of $M_T$ hence it is 
semi-ample. Now $(W'',C'')$ is a good log minimal model of both $(W,C)$ and $(X,B)$.\\ 
\end{proof}

\begin{prop}\label{p-m-model}
Let $(X,\Delta+C)$ be a projective $\Q$-factorial klt pair where $\Delta,C\ge 0$ are $\Q$-Cartier. 
Let $f\colon X\bir Z$ be a rational map onto a normal projective variety 
$Z$ so that we have:

$\bullet$ $f$ is a projective morphism with connected fibres over some non-empty open subset $U\subseteq Z$,  

$\bullet$ $K_X+\Delta\sim_\Q 0$ over $U$, and $C=0$ over $U$, 

$\bullet$ $K_X+\Delta+C$ is nef, and 

$\bullet$ $\kappa(K_X+\Delta)\ge \dim Z$.\\\\ 
Then, for any real number $0\le t\le 1$,  
 we can run an LMMP on $K_X+\Delta+tC$ with scaling of $(1-t)C$ 
which terminates with a good log minimal model of $(X,\Delta+tC)$.
\end{prop}
\begin{proof}
We will show that we can run an LMMP on $K_X+\Delta$ with scaling of $C$ 
which terminates. This in particular shows that $(X,\Delta+tC)$ has a log minimal model 
for each $t\in [0,1]$. That fact that the log minimal model is good will follow from the construction and 
Propositions \ref{p-fg} and \ref{p-klt-trivial} and the finite generation result of 
[\ref{BCHM}]. 

Put $X_1:=X$, $\Delta_1:=\Delta$, and $C_1:=C$. Let $\lambda_1\ge 0$ be the 
smallest number such that $K_{X_1}+\Delta_1+\lambda_1 C_1$ is nef. By [\ref{B-m-model}, Lemma 3.1], 
$\lambda_1$ is a rational number. We may assume that $\lambda_1>0$. 
Since $R(K_{X_1}+\Delta_1+\lambda_1 C_1)$ is a finitely generated $k$-algebra, 
and since $K_{X_1}+\Delta_1+\lambda_1 C_1$ is nef, by Proposition \ref{p-fg}, 
$K_{X_1}+\Delta_1+\lambda_1 C_1$ is semi-ample hence it defines a contraction $X_1\to V_1$. 
Note that $V_1$ is birational to $Z$ since $\kappa(K_{X_1}+\Delta_1+\lambda_1 C_1)=\dim Z$ 
which can be seen as in the proof of Proposition \ref{p-fg}. In particular, 
$K_{X_1}+\Delta_1\sim_\Q 0$ and $C_1=0$ over some non-empty open subset of $V_1$.

Run the LMMP$/V_1$ on $K_{X_1}+\Delta_1$ with scaling of 
an ample$/V_1$ divisor. This terminates 
with a good log minimal model $X_2$ of $(X_1,\Delta_1)$ over $V_1$ by [\ref{B-lc-flips}, Theorem 1.4].  
So, $K_{X_2}+\Delta_2$ is semi-ample$/V_1$.  
Now 
since $K_{X_2}+\Delta_2+\lambda_1 C_2$ is the pullback of some ample divisor on $V_1$, 
$$
K_{X_2}+\Delta_2+\lambda_1 C_2+\delta (K_{X_2}+\Delta_2)
$$ 
is semi-ample for some sufficiently small $\delta>0$. 
In other words, $K_{X_2}+\Delta_2+\tau C_2$ is semi-ample for some  
$\tau<\lambda_1$. 
We can consider $X_1\bir X_2$ as a partial 
LMMP on $K_{X_1}+\Delta_1$ with scaling of $ \lambda_1C_1$.

We can continue the process. That is, let $\lambda_2\ge 0$ be the smallest number such that 
$K_{X_2}+\Delta_2+\lambda_2 C_2$ is nef, and so on (note that $\lambda_1>\tau\ge \lambda_2$).
 This process is an LMMP on $K_X+\Delta$ with scaling of $C$. 
 The numbers $\lambda_i$ that appear in the LMMP satisfy $\lambda:=\lim_{i\to \infty}\lambda_i\neq \lambda_j$ 
for any $j$. The LMMP terminates by [\ref{B-lc-flips}, Theorem 1.9] 
if we show that $(X,\Delta+\lambda C)$ has a log minimal model. 

Remember that $K_{X_2}+\Delta_2$ is semi-ample$/V_1$ hence it defines a contraction 
$X_2\to W_2/V_1$ so that $K_{X_2}+\Delta_2\sim_\Q 0/W_2$. Moreover, by construction, 
$W_2\to V_1$ is birational and 
$$
K_{X_2}+\Delta_2+\lambda_1 C_2\sim_\Q 0/W_2
$$ 
Therefore, 
$$
K_{X_2}+\Delta_2+\lambda C_2\sim_\R 0/W_2
$$ 
and by Proposition \ref{p-klt-trivial}, 
$({X_2},\Delta_2+\lambda C_2)$ has a good log minimal model which is also a good log minimal model 
of $(X,\Delta+\lambda C)$.
\end{proof}

\vspace{0.5cm}
\section{Preparations for the proof of Theorem \ref{t-main}}

In this section, we give the necessary preparations for the proof of Theorem \ref{t-main}.
We recommend the reader to read the sketch of proof of Theorem \ref{t-main} given in the 
introduction before continuing.

\begin{lem}\label{l-1}
Let $(X,B)$ and $f$ be as in Theorem \ref{t-main} with the extra assumption that every lc centre 
of $(X,B)$ is vertical$/Z$. Then, we can replace $(X,B)$ and $f$ so that we have the following 
additional properties: 

$\bullet$ $(X,B)$ is $\Q$-factorial dlt, and $f$ is a contraction,

$\bullet$ $M_Z\sim_\Q A_Z+L_Z$ where $A_Z\ge 0$ is ample and $L_Z\ge 0$, 

$\bullet$ $\Supp (A_Z+L_Z)$ does not contain the image of any lc centre of $(X,B)$,

$\bullet$ Letting $A=f^*A_Z$, $L=f^*L_Z$, we can write $A\sim_\Q G+P$ where 
$G\ge 0$ is semi-ample containing no lc centre of $(X,B)$,  

$\bullet$ $P\ge 0$, $\Supp P=\Supp \rddown{B}$, and 

$\bullet$ $K_X+B+ b G+aL$ is nef if $b\gg a\ge 0$. 
\end{lem}
\begin{proof}
We can take a dlt blowup and assume that $(X,B)$ is $\Q$-factorial dlt. 
Since $K_X+B\sim_\Q f^*M_Z$ and ${\bf{B_+}}(M_Z)$ does not contain the image
of any lc centre of $(X,B)$, we can write $M_Z\sim_\Q R_Z+S_Z$ where 
$R_Z\ge 0$ is ample, $S_Z\ge 0$, and $\Supp (R_Z+S_Z)$ does not contain the image 
of any lc centre of $(X,B)$. 
By taking the Stein factorization we may assume 
$f$ is a contraction.

On the other hand, if $0<\epsilon\ll 1$ is a rational number, then 
$$
K_X+\Delta:=K_X+B-\epsilon \rddown{B}
$$ 
is klt and $K_X+\Delta\sim_\Q 0$ 
over some non-empty open subset $U$ of $Z$ because $\rddown{B}$ is vertical$/Z$.
Thus, by [\ref{B-lc-flips}, Theorem 1.4] 
we can run an LMMP$/Z$ on $K_X+\Delta$ ending up with a good 
log minimal model ${X'}$ over $Z$. Let ${f}'\colon {X'}\to {Z'}/Z$ 
be the contraction associated to $K_{{X'}}+{\Delta'}$ 
and write $K_{{X'}}+{\Delta'}\sim_\Q {f'}^*N_{{Z'}}$ for some $N_{{Z'}}$.  
Since $K_X+B\sim_\Q 0/Z$, $K_{{X'}}+{B'}$ is lc. Moreover, 
$$
K_{{X'}}+{B'}\sim_\Q \epsilon\rddown{{B'}}+{f'}^*N_{{Z'}}
$$
Now let $(X'',B'')$ be a dlt blowup of $({{X'}},{B'})$ and let $g\colon X''\to {X'}$ be  
the corresponding morphism. Let $Q''=g^*\epsilon\rddown{{B'}}$ and let 
$N''=g^*f'^*N_{{Z'}}$. Then, 
$$
K_{X''}+B''\sim_\Q Q''+N''
$$ 
and since 
$({{X'}},{B'}-\epsilon \rddown{{B'}})$ is klt, $\Supp Q''=\Supp \rddown{B''}$. 

Since $R_Z$ is ample and $N_{{Z'}}$ 
is ample$/Z$, if $\delta>0$ is a sufficiently small rational number, then 
$R''+\delta N''$ is semi-ample where $R''$ is the pullback of $R_Z$. Let 
$S''$ be the pullback of $S_Z$. Now  from 
$$
Q''+N''\sim_\Q K_{X''}+B''\sim_\Q R''+S''
$$ 
we obtain  
\begin{equation*}
\begin{split}
K_{X''}+B'' & \sim_\Q  \frac{1}{1+\delta} ( R''+ S''+\delta Q''+\delta N'')\\
& = \frac{\delta}{1+\delta} Q''+\frac{1}{1+\delta}( R''+\delta N'')+\frac{1}{1+\delta}S''
\end{split}
\end{equation*}

By putting $P'':=\frac{\delta}{1+\delta} Q''$, taking a general $G''\sim_\Q \frac{1}{1+\delta}( R''+\delta N'')$, and 
letting $L'':=\frac{1}{1+\delta}S''$, we get 
$$
K_{X''}+B''\sim_\Q P''+G''+L''
$$ 

Let $L_Z=\frac{1}{1+\delta}S_Z$ and 
let $A_Z\sim_\Q M_Z-L_Z$ be general. Then, $A_Z$ is ample and 
$\Supp (A_Z+L_Z)$ does not contain the image of any lc centre of $(X'',B'')$. 
Moreover, if we let $A''$ be the pullback of $A_Z$, then by construction, 
$A''\sim_\Q P''+G''$, $G''$ is semi-ample whose support does not contain 
any lc centre of $(X'',B'')$, and $\Supp P''=\Supp \rddown{B''}$. Finally, 
since $G''$ is the pullback of an ample divisor on  ${Z'}$ and 
since $K_{X''}+B''$ and $L''$ are also pullbacks of certain divisors on ${Z'}$, 
it is clear that $K_{X''}+B''+ bG''+aL''$ is nef if $b\gg a\ge 0$.
 Now  
replace $(X,B)$ with $(X'',B'')$ and put $P=P''$, $G=G''$, $A=A''$, and $L=L''$.\\
\end{proof}

\begin{lem}\label{l-2}
Assume that $(X,B)$ and $f$ satisfy the assumptions and the properties listed in Lemma \ref{l-1}. 
Let $0<\epsilon\ll 1$ and $\alpha\gg 0$ be rational numbers 
so that $K_X+B+\epsilon G+\epsilon L+\alpha G$ is nef. Then, we can run an LMMP on 
$K_X+B+\epsilon G+\epsilon L$ with scaling of $\alpha G$ which terminates 
with a good log minimal model of $(X,B+\epsilon G+\epsilon L)$.
\end{lem}
\begin{proof}
Since $0<\epsilon\ll 1$, $G$ is semi-ample, and $\Supp L$ does not contain any lc centre of 
$(X,B)$, we can assume that $(X,B+\epsilon G+\epsilon L+\alpha G)$ is dlt. 
On the other hand, since 
$$
K_X+B\sim_\Q P+G+L
$$ 
we can write 
 \begin{equation*}
\begin{split}
K_X+B+\epsilon G+\epsilon L & \sim_\Q K_X+B+\epsilon G+\epsilon L+\epsilon P-\epsilon P\\
 & \sim_\Q (1+\epsilon)(K_X+B)-\epsilon P\\
 & \sim_\Q (1+\epsilon)(K_X+B-\frac{\epsilon}{1+\epsilon}P)
\end{split}
\end{equation*}
Thus, it is enough to run an LMMP on $K_X+\Delta:=K_X+B-\frac{\epsilon}{1+\epsilon}P$ with scaling 
of $\beta G$ where $\beta=\frac{\alpha}{1+\epsilon}$. Since $\Supp P=\Supp \rddown{B}$, 
we can assume that $(X,\Delta+\beta G)$ is klt. Moreover, there is a non-empty open subset $U\subset Z$ 
such that $K_X+\Delta\sim_\Q 0$ over $U$ and $G=0$ over $U$.
 Now apply Proposition \ref{p-m-model}.\\
\end{proof}

\begin{lem}\label{l-4}
Assume that $(X,B)$ and $f$ satisfy the assumptions and the properties listed in Lemma \ref{l-1}. 
Assume that $X'$ is a log minimal model of 
$(X,B+\epsilon G+\epsilon L)$ for some rational number $0<\epsilon\ll 1$ obtained as in Lemma \ref{l-2}. 
Then, we can run an LMMP on $K_{X'}+B'$ with scaling of $\epsilon G'+\epsilon L'$ such that 
$\lim \lambda_i=0$ where $\lambda_i$ are the numbers appearing in the LMMP with scaling.
\end{lem}
\begin{proof}
Let $\epsilon=a_1>a_2>a_3>\cdots$ be a strictly decreasing sequence of rational 
numbers approaching zero.  As in the proof of Lemma \ref{l-2}, we 
can write 
$$
K_{X}+B+a_2 G+a_2 L\sim_\Q r(K_{X}+\Delta)
$$ 
and 
$$
K_{X}+B+\epsilon G+\epsilon L\sim_\Q r(K_{X}+\Delta+\tau G+\tau L)
$$
where $r,\tau>0$ are rational, $\Delta\ge 0$, and $(X,\Delta+\tau G+\tau L)$ and 
$(X',\Delta'+\tau G'+\tau L')$ are klt.

By Lemma \ref{l-2}, $X\bir X'$ is obtained by 
an LMMP on $A+L+\epsilon G+\epsilon L$ with scaling of some $\alpha G$. 
It is also an LMMP on $A+L+\epsilon L$.
Let $U\subset Z$ be the complement of $\Supp (A_Z+L_Z)\cup f(G)$. 
Then, $X\bir X'$ is an isomorphism when restricted to $f^{-1}U$. So, 
$X'\bir Z$ is a projective morphism with connected fibres over $U$.
Moreover, $G=0=L$ and  
$K_X+\Delta\sim_\Q 0$ over $U$ which implies that $G'=0=L'$ and 
$K_{X'}+\Delta'\sim_\Q 0$  over $U$. 
In addition, 
$$
\kappa(K_{X'}+\Delta')\ge \kappa(K_{X}+\Delta)=\dim Z
$$ 
where the inequality follows from the fact that $X'\bir X$ does not contract divisors and 
$K_{X'}+\Delta'$ is the pushdown of $K_{X}+\Delta$ (for each sufficiently divisible integer 
$m$ we have an inclusion 
$H^0(X,m(K_{X}+\Delta))\subseteq H^0(X',m(K_{X'}+\Delta')$).
Therefore, by Proposition \ref{p-m-model}, 
we can  run an LMMP on $K_{X'}+\Delta'$ with scaling of 
$\tau G'+\tau L'$ which terminates. This corresponds to an LMMP on 
$K_{X'}+B'+a_2 G'+a_2 L'$ with scaling of $(a_1-a_2)G'+(a_1-a_2)L'$ which terminates on 
some model $X''$.

Next, using the same arguments as above we can run an LMMP on 
$K_{X''}+B''+a_3 G''+a_3 L''$ with scaling of $(a_2-a_3)G''+(a_2-a_3)L''$ which terminates. 
Continuing this process gives the desired LMMP
on $K_{X'}+B'$ with scaling of $\epsilon G'+\epsilon L'$ such that 
$\lim \lambda_i=0$ where $\lambda_i$ are the numbers appearing in the LMMP.\\
\end{proof}

The next lemma will be used to do induction on dimension in the proof of Theorem \ref{t-main}.

\begin{lem}\label{l-3}
Assume that Theorem \ref{t-main} holds in dimension $\le d-1$. 
Let $(X,B)$ and $f$ satisfy the assumptions 
and the properties listed in Lemma \ref{l-1} where $d=\dim X$. Let $U\subset Z$ be a non-empty open set   
and $\phi\colon X\bir X'$ a birational map satisfying: 

$\bullet$ $\phi^{-1}$ does not contract divisors,

$\bullet$  $\phi|_{f^{-1}U}$ is an isomorphism,

$\bullet$ the generic point of each lc centre of $(X,B)$ is in ${f^{-1}U}$,

$\bullet$ $({X'},B')$ is $\Q$-factorial dlt,

$\bullet$ the generic point of each lc centre of $({X'},B')$ is in $\phi({f^{-1}U})$.\\\\ 
Let $S$ be an lc centre of $(X,B)$, $S'$ its birational transform on $X'$, and by adjunction 
define $K_{S'}+B_{S'}':=(K_{X'}+B')|_{S'}$. Then, $(S',B_{S'}')$ has a good log minimal model.  
\end{lem}
\begin{proof}
Let $\psi\colon S\bir S'$ be the induced birational map. Let  
$g\colon S\to V$ be the contraction given by the Stein factorization of 
$S\to f(S)$, and let $U_V\subset V$ be the inverse image of $U$. 
By assumptions, $\psi$ is an isomorphism when restricted to $g^{-1}U_V$.
Moreover, by construction, the generic point of each lc centre of $(S',B'_{S'})$ is 
inside $\psi(g^{-1}U_V)$. 

Let $W\to X$ and $W\to X'$ be a common log resolution so that it induces a 
common log resolution $h\colon T\to S$ and $e\colon T\to S'$ where $T\subset W$ is the birational 
transform of $S$. Such $W$ exists because $({X},B)$ is $\Q$-factorial dlt and 
$X\bir X'$ is an isomorphism near the generic point of $S$.
Let $B_T'$ be a rational boundary so that $(T,B_T')$ is a log smooth model of $(S',B_{S'}')$, as 
defined in \ref{d-mws},  
and so that each lc centre of $(T,B_T')$ maps onto a lc centre of $(S',B_{S'}')$. 
This ensures 
that the generic point of each lc centre of $(T,B_T')$ is mapped into $U_V$.
 
Now run an LMMP$/V$ on $K_T+B_T'$ with scaling of an ample divisor. By adjunction define 
$K_S+B_S=(K_X+B)|_S$. Since $\psi$ is an isomorphism on $g^{-1}U_V$ and since 
$K_S+B_S\sim_\Q 0/V$, the pair $(S',B_{S'}')$ is a weak lc model of $(T,B_T')$ 
over $U_V$. Thus, the LMMP terminates over $U_V$ by 
[\ref{B-lc-flips}, Corollary 3.7 and Theorem 1.9] and
we reach a model $\overline{T}$ on which $K_{\overline{T}}+B_{\overline{T}}'\sim_\Q 0$ 
over $U_V$. Moreover, the generic point of each lc centre of 
$({\overline{T}},B_{\overline{T}}')$ is mapped into $U_V$. Therefore,  by 
[\ref{B-lc-flips}, Theorem 1.4], we can run an LMMP$/V$ on  
$K_{\overline{T}}+B_{\overline{T}}'$ which 
terminates with a good log minimal model of $(T,B_T')$ over $V$. Replacing $\overline{T}$ with the minimal model 
we may assume that $K_{\overline{T}}+B_{\overline{T}}'$ is 
semi-ample$/V$. Let $\overline{g}\colon \overline{T}\to \overline{V}/V$ be the contraction 
defined by $K_{\overline{T}}+B_{\overline{T}}'$. 

Replacing $M_Z$ by $A_Z+L_Z$ enable us to assume that $M_Z=A_Z+L_Z$. Then, 
$$
K_S+B_S\sim_\Q g^*M_V=g^*(A_V+L_V)
$$ 
where $M_V$, $A_V$, and $L_V$ are the pullbacks on $V$ of 
$M_Z$, $A_Z$, and $L_Z$ respectively. 
On the other hand, 
$$
K_{S'}+B'_{S'}=(K_{X'}+B')|_{S'}\sim_\Q (A'+L')|_{S'}
$$ 
and 
$$
K_T+B_T'=e^*(K_{S'}+B'_{S'})+E_T
$$ 
for some $E_T\ge 0$ which is exceptional$/S'$. Writing 
$$
M_T'=e^*(A'+L')|_{S'}+E_T
$$
we have 
$$
K_T+B_T'\sim_\Q M_T' ~~~\mbox{and}~~~ K_{\overline{T}}+B_{\overline{T}}'\sim_\Q M_{\overline{T}}'
$$

The following diagram shows some of the 
objects and maps we have 
constructed so far:
$$
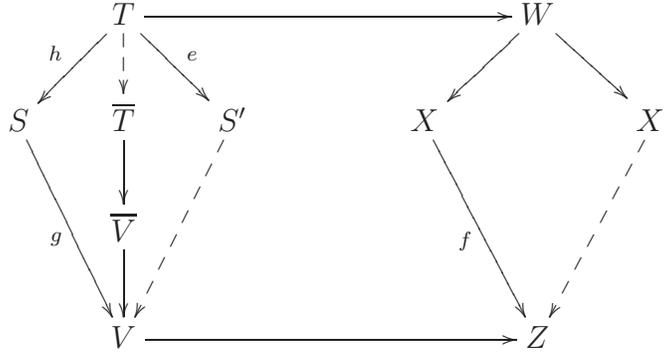
\xymatrix{
 & T\ar[ld]_h\ar[rd]^e\ar@{-->}[d] \ar[rrrr] &          &&     & W\ar[rd]\ar[ld] &   \\
S\ar[rdd]_g & \overline{T}\ar[d] & S'\ar@{-->}[ldd] &&   X \ar[rdd]_f&   & X'\ar@{-->}[ldd]\\
&\overline{V}\ar[d]&   &&&&\\
&V \ar[rrrr]&                                       &&     & Z &
} 
$$

By construction,  
$E_{\overline{T}}$ is mapped into $V\setminus U_V$ since the above LMMP contracts 
any component of $E_T$ whose generic point is mapped into $U_V$. Moreover,
over $U_V$, $M_{\overline{T}}'$ is nothing but the pullback of $M_V=A_V+L_V$. 
Therefore, we can write  
$$
K_{\overline{T}}+B_{\overline{T}}'\sim_\Q M_{\overline{T}}'= \overline{g}^*M_{\overline{V}}'
$$ 
such that $M_{\overline{V}}'\ge 0$ 
and $M_{\overline{V}}'=p^*M_V$ over $U_V$ if we denote $\overline{V}\to V$ by $p$. 
Note that $p$ is birational and it is an isomorphism over $U_V$.

Now let $A_W$ on $W$ be the pullback of $A$ and similarly let $A_W'$ be the 
pullback of $A'$. Since $A_W$ is nef and $\phi^{-1}$ does not contract divisors, 
by the negativity lemma we have $A_W'\ge A_W$. By construction, 
$M_T'\ge A_W'|_T$ hence $M_T'\ge A_W|_T$ which in turn implies that 
$M_{\overline{T}}'\ge \overline{g}^*p^*A_V$ so $M_{\overline{V}}'\ge p^*A_V$. 
Let $N_{\overline{V}}'=M_{\overline{V}}'-p^*A_V$. Over $U_V$ we have 
$N_{\overline{V}}'=p^*L_V$, and the generic point of each lc centre of 
$({\overline{T}},B_{\overline{T}}')$ is mapped into $U_V$ 
hence $\Supp N_{\overline{V}}'$ does not contain the image 
of any lc centre of $({\overline{T}},B_{\overline{T}}')$.

Finally, by construction, $M_{\overline{V}}'$ is ample$/V$ hence $N_{\overline{V}}'$ is 
also ample$/V$. So, we can write 
$M_{\overline{V}}'\sim_\Q A_{\overline{V}}'+L_{\overline{V}}'$
where $A_{\overline{V}}'\sim_\Q p^*A_V+\delta N_{\overline{V}}'$ is ample and  
$L_{\overline{V}}'=(1-\delta)N_{\overline{V}}'$ for some small rational number $\delta>0$. 
Choosing $A_{\overline{V}}'$ general 
makes sure that 
$\Supp (A_{\overline{V}}'+L_{\overline{V}}')$ hence ${\bf{B_+}}(M_{\overline{V}}')$ 
does not contain the image 
of any lc centre of $({\overline{T}},B_{\overline{T}}')$.
Since we are assuming Theorem \ref{t-main} in dimension $<d$, 
the pair $({\overline{T}},B_{\overline{T}}')$ has a good log minimal model which is also a good log minimal 
model of $(S',B_{S'}')$.
\end{proof}

\vspace{0.5cm}
\section{Proof of Main results}

In this section, we will prove the main theorem and its corollaries.

\begin{proof}(of Theorem \ref{t-main})
Assume that the theorem holds in dimension $\le d-1$. 
Let $(X,B)$ and $f$ be as in the theorem with $\dim X=d$. 
We can take a dlt blowup hence assume that $(X,B)$ is $\Q$-factorial dlt. 
Moreover, by taking the Stein factorization of $f$ we can assume that $f$ is a 
contraction.

First assume that there is a lc centre of $(X,B)$ mapping onto $Z$. In this case, there is a 
component $S$ of $\rddown{B}$ such that $S$ is mapped onto $Z$. Let 
$g\colon S\to Z$ be the induced morphism. By adjunction 
define $K_S+B_S=(K_X+B)|_S$. Then, $(S,B_S)$ and $g$ satisfy the assumptions of 
Theorem \ref{t-main} since $K_S+B_S\sim_\Q g^*M_Z$ and no lc centre of 
$(S,B_S)$ is mapped into ${\bf{B_+}}(M_{Z})$. 
By induction, $(S,B_S)$ has a good log minimal model. 
In particular, the algebra $R(K_S+B_S)$ is finitely generated over $k$. By [\ref{B-lc-flips}, Lemma 6.4], 
the algebra $R(M_Z)$ is also finitely generated which in turn implies that the algebra 
$R(K_X+B)$ is finitely generated as $f$ is a contraction. Now by Proposition \ref{p-fg}, $(X,B)$ has a good 
log minimal model. So from now on we assume that every lc centre of $(X,B)$ 
is vertical$/Z$. We will construct a log minimal model and use Lemma \ref{l-good-m-model} below 
to deduce the existence of a good log minimal model.

We can replace $(X,B)$ and $f$ so that they satisfy the properties listed in Lemma \ref{l-1}. 
Let $\epsilon_1>\epsilon_2>\cdots $ be a sequence of sufficiently small rational 
numbers with $\lim \epsilon_j=0$. By Lemma \ref{l-2}, for each $\epsilon_j$, 
we can run an LMMP on 
$K_X+B+\epsilon_j G+\epsilon_j L$ with scaling of $\alpha G$ for some large $\alpha$ such that 
the LMMP ends up with a good log minimal model $(X_j',B_j'+\epsilon_j G_j'+\epsilon_j L_j')$ 
of $(X,B+\epsilon_j G+\epsilon_j L)$. 
Any prime divisor contracted by 
$X\bir X_j'$ is a component of $A+L$ since the LMMP is also an LMMP on $A+L+\epsilon_j L$. 
Thus, perhaps after replacing the sequence $\epsilon_1, \epsilon_2,\dots$ with a subsequence, 
we could assume that the maps $X\bir X_j'$ contract 
the same prime divisors hence assume that the $X_j'$ are all isomorphic in codimension one. 
This also implies that $K_{X_1'}+B_1'$ is a limit of movable divisors.

Let $X':=X_1'$, $B':=B_1'$, $G':=G_1'$, and $L':=L_1'$, etc. By Lemma \ref{l-4}, we can run an LMMP on $K_{X'}+B'$ 
with scaling of $\epsilon_1 G'+\epsilon_1 L'$ such that 
$\lim \lambda_i=0$ where $\lambda_i$ are the numbers appearing in the LMMP. The LMMP does not contract 
any divisors because $K_{X'}+B'$ is a limit of movable divisors. For each $\epsilon_j$, there is a 
model $Y$ which appears in some step of the LMMP on $K_{X'}+B'$ and some $i$ such that 
$\lambda_i\epsilon_1\ge \epsilon_j\ge \lambda_{i+1}\epsilon_1$ and such that 
$$
K_Y+B_Y+\lambda_i\epsilon_1 G_Y+\lambda_i\epsilon_1 L_Y
$$ 
and 
$$
K_Y+B_Y+\lambda_{i+1} \epsilon_1 G_Y+\lambda_{i+1} \epsilon_1 L_Y
$$ 
are both nef. Now since $X_j'\bir Y$ is an isomorphism in codimension one, 
$(Y,B_Y+\epsilon_j G_Y+\epsilon_j L_Y)$ is also a log minimal model 
of $(X,B+\epsilon_j G+\epsilon_j L)$ so by replacing $X_j'$ with $Y$ we could assume 
that each of the $X_j'$ appears as a model in the LMMP on $K_{X'}+B'$. By redefining the notation 
we assume that the LMMP on $K_{X'}+B'$ is as $X'=X_1'\bir X_2'\bir X_3'\bir \cdots$. 

Assume that the LMMP on $K_{X'}+B'$ terminates. Then, we arrive at a model $X''$ on which  
all the divisors $K_{X''}+B''+\epsilon_j G''+\epsilon_j L''$ are semi-ample when $j\gg 0$. 
For any prime divisor $D$ on $X$ we have 
$$
a(D,X,B+\epsilon_j G+\epsilon_j L)\le a(D,X'',B''+\epsilon_j G''+\epsilon_j L'')\le a(D,X'',B'')
$$
hence taking limit gives
$$
a(D,X,B)\le a(D,X'',B'')
$$
Therefore, $(X'',B'')$ is a weak lc model of $(X,B)$ from which we can construct a log minimal model 
by [\ref{B-lc-flips}, Corollary 3.7]. So, it is enough to show that the LMMP terminates.

First we will show that the LMMP terminates near $\rddown{B'}$. Let $U=Z\setminus \Supp(A_Z+L_Z)$. 
By assumptions, the generic point of each lc centre of $(X,B)$ belongs to $f^{-1}U$. 
Since $X\bir X'$ is an 
LMMP on $A+L+\epsilon_1 L$, it is an isomorphism when restricted to  
$f^{-1}U$. On the other hand, since $X'=X_1'\bir X_i'$ is an LMMP on $A'+L'$, it 
is an isomorphism when restricted to the image of $f^{-1}U$ in $X'$. So, 
the rational maps $X\bir X_i'$ are all isomorphisms when restricted to 
$f^{-1}U$. Moreover, the generic point of each lc centre of $(X_i',B_i')$ belongs 
to the image of $f^{-1}U$ in $X_i'$.
This enables us to use Lemma \ref{l-3}. More precisely, 
let $S$ be a lc centre of 
$(X,{B})$ and $S_i'$ its birational transform on $X_i'$.
By adjunction define $K_{S_i'}+B_{S_i'}'=(K_{X_i'}+B_{i}')|_{S_i'}$. By Lemma \ref{l-3}, 
$({S_i'},B_{S_i'}')$ has a good log minimal model.

 The LMMP  on $K_{X'}+B'$ does not contract $S'=S_1'$. If $S'$ is a minimal lc centre, then by 
 [\ref{pl-flips}][\ref{Fujino-s-term}], the induced birational maps $S_i'\bir S_{i+1}'$ 
 are isomorphisms in codimension one for $i\gg 0$ and 
the LMMP induces an LMMP with scaling on some dlt blowup of $({S'_j},B'_{S_j'})$ for some $j$ 
(see [\ref{B-lc-flips}, Remark 2.10] for more information on this kind of reduction). Since  
$({S'_j},B'_{S_j'})$ has a log minimal model, the LMMP terminates by [\ref{B-lc-flips}, Theorem 1.9].
If $S'$ is not a minimal lc centre, by induction, we can assume that the LMMP on $K_{X'}+B'$ terminates near any lc centre of 
$({S'_j},B'_{S_j'})$ for any large $j$. Thus, again by [\ref{pl-flips}][\ref{Fujino-s-term}], 
the induced birational maps $S_i'\bir S_{i+1}'$ 
 are isomorphisms in codimension one for $i\gg 0$ and 
we may assume that we get an induced LMMP with scaling on some dlt blowup of  
$({S'_j},B'_{S_j'})$ for some $j$. The latter LMMP terminates 
for the same reasons as before. Therefore, we can assume that the 
LMMP on $K_{X'}+B'$ terminates near $\rddown{B'}$.

Outside $\rddown{B'}$ we have 
$$
G'+L'=P'+G'+L'\sim_\Q A'+L'\sim_\Q K_{X'}+B'
$$
Therefore, the LMMP on $K_{X'}+B'$ terminates everywhere otherwise   
the extremal ray in each step of the LMMP intersects $K_{X'}+B'$ negatively but 
intersects $G'+L'$ positively which is a contradiction. 
This completes the proof of the theorem.\\
\end{proof}

\begin{lem}\label{l-good-m-model}
Assume that Theorem \ref{t-main} holds in dimension $\le d-1$.
Let $(X,B)$ and $f$  satisfy the assumptions and properties listed in Lemma \ref{l-1} 
where $d=\dim X$. Then, any log minimal model of $(X,B)$ is good. 
\end{lem}
\begin{proof}
Assume that $(X,B)$ has a log minimal model. It is enough to show that $(X,B)$ has 
one good log minimal model because then all the other log minimal models would be good. 
By [\ref{B-lc-flips}, Theorem 1.9], 
we can run an LMMP on $K_X+B$ 
with scaling of some ample divisor which ends up with a log minimal model $X'$. 
The LMMP is an LMMP on $A+L$. 
Let $U=Z\setminus \Supp(A_Z+L_Z)$. Then, the rational map $\phi\colon X\bir X'$ 
is an isomorphism when restricted to $f^{-1}U$. Moreover, the generic point of 
each lc centre of $(X,B)$ belongs to $f^{-1}U$ which in turn implies that the generic point of 
each lc centre of $(X',B')$ belongs to $\phi (f^{-1}U)$.
In view of 
$$
\kappa(K_{X'}+B')=\kappa(K_X+B)=\kappa_\sigma(K_X+B)=\kappa_\sigma(K_{X'}+B')
$$
and Lemma \ref{l-3}, 
the pair $(X',B')$ is log abundant. Therefore, by [\ref{Fujino-Gongyo}, Theorem 4.2], 
$K_{X'}+B'$ is semi-ample hence $(X',B')$ is a good log minimal model.\\
\end{proof}

\begin{proof}(of Corollary \ref{cor-1})
This follows from Theorem \ref{t-main} by taking $Z=X$ and $f$ to be the identity morphism.
\end{proof}

The next example shows that in Corollary \ref{cor-1} we cannot simply weaken the assumption 
that ${\bf{B_+}}(K_X+B)$ does not contain any lc centre of $(X,B)$ to  
$(X,B)$ being log big.

\begin{exa}\label{exa-log-big}
 Take a smooth projective variety $S$ with $\kappa(K_S)\ge 0$, let $Z$ be the projective cone 
over $S$ (with respect to some very ample divisor) and $X_2\to Z$ the blowup of the vertex. Identify $S$ with the 
exceptional divisor of $X_2\to Z$. Pick a smooth ample divisor $H$ on 
$S$ and let $\pi\colon X=X_1\to X_2$ be the blowup of $X_2$ along $H$. Now let 
$B=T+\epsilon E+\frac{1}{2}A$ where $T$ is the birational transform of $S$, $E$ is 
the exceptional divisor of $X\to X_2$, $\epsilon>0$ is small,  
and $A$ is the pullback of a sufficiently ample divisor on $Z$. Note that $T\to S$  
and $E\cap T\to H$ are isomorphisms hence $(K_X+B)|_T$ is big as it is identified with 
$K_S+\epsilon H$ via $T\to S$.  
Therefore, $(X,B)$ is log big: indeed $K_X+B$ is big, $(K_X+B)|_T$ is big, and 
$T$ is the only lc centre of $(X,B)$. 
Now run an LMMP 
on $K_X+B$ with scaling of some ample divisor. The LMMP would necessarily be over $Z$ because 
of the presence of $A$. It is quite likely that the first step of the LMMP 
is the contraction $X\to X_2$. In that case termination of the LMMP is essentially 
equivalent to existence of a minimal model of $S$. 
Even if $K_S$ is already nef, showing that $K_{X_2}+B_2$ is semi-ample amounts to showing 
semi-ampleness of $K_S$ which is the abundance problem. 
\end{exa}

\begin{proof}(of Corollary \ref{cor-2})
We can write 
$$
K_X+B+P\sim_\Q A+L
$$ 
where $A\ge 0$ is ample, $L\ge 0$, and $\Supp(A+L)$ does not contain any lc 
centre of $(X,B)$. For some small rational number $\epsilon>0$, we can write 
$$
K_X+\Delta\sim_\Q K_X+B+P+\epsilon (A+L)\sim_\Q (1+\epsilon)(K_X+B+P)
$$
such that $(X,\Delta)$ is lc and any lc centre of $(X,\Delta)$ is also a lc centre of 
$(X,B)$. The augmented base locus
$$
{\bf{B_+}}(K_X+\Delta)={\bf{B_+}}((1+\epsilon)(K_X+B+P))={\bf{B_+}}(K_X+B+P)
$$
does not contain any lc centre of $(X,\Delta)$. 
Now apply Corollary \ref{cor-1} to deduce that $R(K_X+\Delta)$ is finitely generated which in turn 
implies that $R(K_X+B+P)$ is also finitely generated.\\ 
\end{proof}


\vspace{2cm}

\flushleft{DPMMS}, Centre for Mathematical Sciences,\\
Cambridge University,\\
Wilberforce Road,\\
Cambridge, CB3 0WB,\\
UK\\
email: c.birkar@dpmms.cam.ac.uk\\
email: zh262@dpmms.cam.ac.uk

\end{document}